\documentclass[11pt]{amsart}
\usepackage{amsfonts}
\usepackage{amsmath}
\usepackage{amsthm}
\usepackage{bm}
\usepackage{bbm}
\usepackage{graphicx}
\usepackage{float}
\usepackage{enumerate}
\usepackage{natbib}

\usepackage{pdfsync}

\usepackage{mathtools}

\usepackage{blindtext}
\usepackage[utf8]{inputenc}

\usepackage[hmarginratio=1:1,top=32mm,columnsep=20pt]{geometry}

\usepackage[T1]{fontenc} 
\usepackage{color}	
\usepackage[english]{babel} 

\usepackage{enumitem} 
\setlist[itemize]{noitemsep} 

\usepackage{leftidx}   

\usepackage{mathrsfs}

 \usepackage[all]{xy}
 
 \newcommand{\rom}[1]{%
  \textup{\lowercase\expandafter{\romannumeral#1}}%
}

\newcommand{\Rom}[1]{%
  \textup{\uppercase\expandafter{\romannumeral#1}}%
}

\newtheorem{theorem}{Theorem}[section]

\newtheorem{lemma}{Lemma}[section]
\newtheorem{proposition}{Proposition}[section]
\newtheorem{remark}{Remark}[section]

\usepackage{enumitem}  

\newcommand{\BX}{{\bf X}}

\newcommand{\bx}{{\bf x}}
\newcommand{\by}{{\bf y}}
\newcommand{\bs}{{\bf s}}
\newcommand{\bt}{{\bf t}}
\newcommand{\bn}{{\bf n}}
\newcommand{\bk}{{\bf k}}

\newcommand{\reals}{{\mathbb R}}
\newcommand{\bbr}{\reals}
\newcommand{\bbn}{{\mathbb N}}

\newcommand{\bbz}{\protect{\mathbb Z}}

\newcommand{\one}{{\bf 1}}
\newcommand{\0}{{\bf 0}}
\newcommand{\eid}{\stackrel{d}{=}}

\newcommand{\binfty}{\bm{\infty}}
\newcommand{\bbeta}{\bm{\beta}}

\newcommand{\SaS}{S$\alpha$S}

\begin{document}
 
\title[Extreme Value Theory for Random Fields]
{Extreme Value Theory for Long Range Dependent Stable Random Fields}

\author{Zaoli Chen}
\address{Department of Mathematics \\
Cornell University \\
Ithaca, NY 14853}
\email{zc288@cornell.edu}

\author{Gennady Samorodnitsky}
\address{School of Operations Research and Information Engineering\\
and Department of Statistical Science \\
Cornell University \\
Ithaca, NY 14853}
\email{gs18@cornell.edu}


\numberwithin{equation}{section}
 \thanks{This research was partially supported by the NSF grant
  DMS-1506783 and the ARO grant
   W911NF-18 -10318 at Cornell  University.}

\subjclass{Primary 60G60, 60G70, 60G52}
\keywords{Random field, extremal limit theorem, random sup measure,
  random closed set, long range dependence, stable law, heavy tails}

\begin{abstract}
We study  the extremes for a class of a symmetric stable random fields 
with long 
range dependence. We prove functional extremal theorems both in the
space of sup measures and in the space of cadlag functions of several
variables. The limits in both types of theorems are of a new kind, and
only in a certain range of parameters these limits have the Fr\'echet
distribution. 
\end{abstract}

\maketitle

\section{Introduction} \label{sec:intro}

Extreme value theorems describe the limiting behaviour of the largest
values in increasingly large collections of random variables. The
classical extremal theorems, beginning with \cite{fisher:tippett:1928}
and \cite{gnedenko:1943}, deal with the extremes of
i.i.d. (independent and identically distributed) random
variables. The modern extreme value theory techniques allow us to
study the extremes of dependent sequences; see
\cite{leadbetter:lindgren:rootzen:1983} and the expositions in
\cite{coles:2001} and \cite{dehaan:ferreira:2006}. The effect of
dependence on extreme values can be restricted to a loss in the
effective sample size, through the extremal index of the sequence. When
the dependence is sufficiently long, more significant changes in
extreme value may occur; see e.g. \cite{samorodnitsky:2004a},
\cite{owada:samorodnitsky:2015a}. The present paper aims to contribute
to our understanding of the effect of memory on extremes when the time
is of dimension larger than 1, i.e. for random fields. 

We consider a discrete time stationary random field $\BX = \bigl(
X_\bt, \, \bt\in\bbz^d\bigr)$. For $\bn=(n_1,\ldots, n_d)\in \bbn^d$ we
would like to study the extremes of the random field over growing
hypercubes of the type
$$
[\0,\bn]=\bigl\{ \0\leq \bk\leq \bn\bigr\}, \ \bn\to\binfty\,,
$$
where $\0$ is the vector with zero coordinates, the notation $\bs\leq
\bt$ for vectors $\bs=(s_1,\ldots, s_d)$ and $\bt=(t_1,\ldots, t_d)$
means that $s_i\leq t_i$ for all $i=1,\ldots,d$, and the notation $\bn\to\binfty$
means that all $d$ components of the vector $\bn$ tend to
infinity. Denote
$$
M_\bn = \max_{ \0\leq \bk\leq \bn}X_\bt\,.
$$
What limit theorems does the array $(M_\bn)$ satisfy? It was shown by
\cite{leadbetter:rootzen:1998}  that
under appropriate strong mixing conditions, only the classical three
types of limiting distributions (Gumbel, Fr\'echet and Weibull) may
appear (even when forcing $\bn\to\binfty$ along a monotone curve). In
the case when the marginal distributions of the field $\BX$ have
regularly varying tails, this allows only the Fr\'echet distribution
as a limit. 

In this paper we will discuss only random fields with regularly
varying tails, in which case the experience from the classical extreme
value theory tells us to look for limit theorems for the type
\begin{equation} \label{e:limit.thm.1d}
\frac{1}{b_\bn}M_\bn\Rightarrow Y \ \ \text{as $\bn\to\binfty$}
\end{equation}
for some nondegenerate random variable $Y$. The regular
variation of the marginal distributions means that 
\begin{equation} \label{e:marg.tail}
P( X(\0)>x)= x^{-\alpha}L(x), \ \alpha>0, \ \text{$L$ slowly varying,}
\end{equation} 
see e.g. \cite{resnick:1987}. 
Notice that the assumption is only on the right tail of the
distribution since, in most cases, one does not expect a limit theorem
for the partial maxima as in \eqref{e:limit.thm.1d} to be affected by
the left tail of $X(\0)$. 

If the random field $\BX$ consists of i.i.d. random variables
satisfying the regular variation condition \eqref{e:marg.tail}, then
the classical extreme value theory tells us that the convergence in
\eqref{e:limit.thm.1d} holds if we choose 
\begin{equation} \label{e:bn.standard}
b_\bn = \inf\bigl\{ x>0:\, P( X(\0)>x) \leq (n_1\cdots
n_d)^{-1}\bigr\}\,,
\end{equation} 
in which case the limiting random variable $Y$ has the standard
Fr\'echet distribution. We are interested in understanding how the
spatial dependence in the random field $\BX$ affects the scaling in 
and the distribution of the limit not only in \eqref{e:limit.thm.1d},
but in its functional versions, which can be stated in different
spaces, for example in the space $D(\mathbb{R}_+^d)$ of right
continuous, with limits along monotone paths, functions (see
\cite{straf:1972}), or in the space of random sup measures
$\mathcal{M}(\mathbb{R}_+^d)$; see
\cite{obrien:torfs:vervaat:1990}. We will describe the relevant spaces
below. 

If the time is one-dimensional, and the memory in the stationary
process is short, then the standard normalization
\eqref{e:bn.standard} is still the appropriate one, and the limits
both in \eqref{e:limit.thm.1d} and its functional versions change only
through a change in a multiplicative constant; see
\cite{samorodnitsky:2016} and references therein. However, when the
memory becomes sufficiently long, both the order of magnitude of the
normalization in the limit theorems changes, and the nature of the
limit changes as well; see \cite{samorodnitsky:2004a} and
\cite{owada:samorodnitsky:2015a}. Furthermore, the limit may even stop 
having the Fr\'echet distribution (or Fr\'echet marginal
distributions, in the functions limit theorems); see
\cite{samorodnitsky:wang:2017}. It is reasonable to expect that
similar phenomena happen for random fields, but because it is harder
to quantify how long the memory is when the time is not
one-dimensional, less is known in this case. 

In this paper we will concentrate on the case where the random field
$\BX$ 
is a symmetric $\alpha$-stable (\SaS) random field,
$0<\alpha<2$. Recall that this 
means that every finite linear combination of the values of the values
of the random field has a one-dimensional \SaS\  distribution, i.e. has
a characteristic function of the form
$\exp\{-\sigma^\alpha|\theta|^\alpha\}$, $\theta\in\bbr$, where
$\sigma\in [0,\infty)$ is a scale parameter that  depends on the linear
combination; see \cite{samorodnitsky:taqqu:1994}. The marginal
distributions of \SaS\ random fields satisfy the regular variation
assumption \eqref{e:marg.tail} with $0<\alpha<2$  that  coincides with
the index of stability. In this case a series of results on the
relation between the sizes of the extremes of stationary \SaS\ random
fiels and certain ergodic-theoretical properties of the L\'evy
measures of these fields is due to Parthanil Roy and his coworkers;
see  \cite{roy:samorodnitsky:2008},
\cite{chakrabarty:roy:2013}, \cite{sarkar:roy:2016}. These results are
made possible because of the connection between the structure of the
\SaS\ random fields and ergodic theory established by
\cite{rosinski:2000}. 

This paper contributes to understanding the extremal limit theorems
for \SaS\ random fields and their connection to the dynamics of the
L\'evy measures. In this sense our paper is related to the ideas of 
\cite{rosinski:2000}. However, we will restrict ourselves to certain
Markov  flows. This will allow us to avoid, to a large extent, the
language of ergodic theory, and state everything in purely
probabilistic terms. There is not doubt, however, that our results
could be extended to more general dynamical systems acting on the
L\'evy measures of \SaS\ random fields. The generality in which work
is sufficient to demonstrate the new phenomena that may arise in
extremal limit theorems for random fields with long range
dependence. We will exhibit new types of limits, some of which will
have non-Fr\'echet distributions, both in the space of random sup
measures and in the space $D(\mathbb{R}_+^d)$. 

This paper is organized as follows. In Section \ref{sec:random.field}
we introduce the class of stationary symmetric $\alpha$-stable random
fields we will study in this paper. In Section \ref{sec:supmeasures}
we provide some background on random closed sets and random sup
measures, and describe the limiting random sup
measure that appears as the weak limit the extremal theorem in Section
\ref{Convergence   of the random sup measures}. Finally, in Section
\ref{Convergence of the partial maxima processes} we prove versions of
our extremal limit theorems in the space $D(\mathbb{R}_+^d)$.

{\bf Notation}: \ For a function $g$ on an arbitrary set with values
in a linear space  we denote  the set of zeros of $g$ by
$\mathcal{Z}(g) $. Arithmetic operations involved vectors are
performed component-wise. Thus, if $\bx=(x_1,\ldots,x_d)$ and
$\by=(y_1,\ldots, y_d)$, then, say, $\bx\by=(x_1y_1,\ldots,
x_dy_d)$. This extends to sets: if $A=A_1\times \cdots\times A_d$, then
$\bx A= x_1A_1\times \cdots \times x_dA_d$. 

\section{A \SaS\ random field with long range dependence}
\label{sec:random.field}

We start with a construction of a family of stationary \SaS\ random
fields, $0<\alpha<2$, whose memory has a natural finite-dimensional
parameterization. It is an extension to random fields of models
considered before in the case of one-dimensional time; see
e.g. \cite{resnick:samorodnitsky:xue:2000}, 
\cite{samorodnitsky:2004a}, 
\cite{owada:samorodnitsky:2015,owada:samorodnitsky:2015a},
\cite{owada:2016} and \cite{lacaux:samorodnitsky:2016}. 

We start with $d$ $\sigma$-finite, infinite measures on $
\bigl(\mathbb{Z}^{\bbn_0} , \mathcal{B}(\mathbb{Z}^{\bbn_0}) \bigr)$
defined by 
\begin{equation}   \label{e:mu.j}
\mu_i : = \sum_{k \in \mathbb{Z} } \pi^{(i)}_k  P^{(i)}_k\,,
\end{equation}
where for $i=1,\ldots, d$, $P^{(i)}_k$ is the law of an irreducible
aperiodic null-recurrent Markov chain  $(Y^{(i)}_n)_{ n \geq 0 }$ on $\bbz$ 
starting at $Y^{(i)}_0=k\in\bbz$. Further, $(\pi^{(i)}_k)_{k \in
  \mathbb{Z} }$ is its unique (infinite) invariant measure satisfying 
$\pi^{(i)}_0= 1$. Given this invariant measure, we can extend the
probability measures $P^{(i)}_k$ from measures on
$\mathbb{Z}^{\bbn_0}$ to measures on  $\bbz^\bbz$ which, in turn,
allows us to extend the measure $\mu_i$ in \eqref{e:mu.j} to
$\bbz^\bbz$ as well. We will keep using the same notation as in
\eqref{e:mu.j}. 

We will work with the product space
$$
(E,\mathcal{E})  = \left( \mathbb{Z}^{\bbz}\times\cdot\times 
  \mathbb{Z}^{\bbz}, \ \mathcal{B}(\mathbb{Z}^{\bbz})
  \times\cdot\times \mathcal{B}(\mathbb{Z}^{\bbz}) \right)
$$
of $d$ copies of $\bigl(\mathbb{Z}^{\bbz} ,
\mathcal{B}(\mathbb{Z}^{\bbz}) \bigr)$, on which we put the
product, $\sigma$-finite, infinite, measure
$$
\mu = \mu_1\times \cdot\times\mu_d\,.
$$

The key assumption is a regular variation assumption on the return
times of the Markov chains $(Y^{(i)}_n)_{ n \geq 0 }, \, i=1,\ldots, d$. For
$\bx=(\ldots, x_{-1},x_0,x_1,x_2\ldots)\in \bbz^{\bbz}$ we define the first return
time to the origin by $\varphi(\bx)= \inf \{ n \geq 1 :\, x_n  = 0
\}$. We assume that for $i=1,\ldots, d$ we have 
\begin{equation} \label{e:return.regvar}
P_0^{(i)}(\varphi >n)\in RV_{-\beta_i}
\end{equation}
for some $0<\beta_i<1$. This implies that 
\begin{align} \label{e:bn.j}
& \mu_i \left(  \{\bx:\, x_k =
  0 \; \text{for some} \; k=0,1,\ldots,n \} \right) \\
\sim& \sum_{k=1}^n
P_0^{(i)}(\varphi >k)  \sim (1-\beta_i)^{-1}
nP_0^{(i)}(\varphi >n)  \in RV_{1-\beta_i} \,.\notag 
\end{align} 
See \cite{resnick:samorodnitsky:xue:2000}.

On $\mathbb{Z}^{\bbz}$ there is a natural left shift operator 
$$ 
T \bigl( (\ldots, x_{-1},x_0,x_1,x_2\ldots) \bigr) =    (\ldots,
x_{0},x_1,x_2,x_3\ldots)\,.
$$
It is naturally extended to a group action of $\bbz^d$ on $E$ as
follows. Writing an element $\bx \in E$ as $\bx=(\bx^{(1)},\ldots,
\bx^{(d)})$ with $\bx^{(i)}=(\ldots, x_{-1}^{(i)},x_0 ^{(i)},x_1
^{(i)},x_2 ^{(i)}\ldots) \bigr)\in \bbz^\bbz$ for $i=1,\ldots, d$, 
we set for $\bn=(n_1,\ldots, n_d)\in \bbz^d$, 
\begin{equation}  \label{eq; multi para action}
T^\bn \bx = ( T^{n_1} \bx^{(1)}, \ldots , T ^ {n_d} \bx ^ { (d) } )\in
E\,.
\end{equation}
Even though we are using the same notation $T$ for operators acting on
different spaces, the meaning will always be clear from the
context. Note that each individual left shift $T$ on
$\bigl(\mathbb{Z}^{\bbn_0} , \mathcal{B}(\mathbb{Z}^{\bbn_0})
,\mu_j\bigr)$ is measure preserving (because each $(\pi^{(j)}_i)_{i
  \in   \mathbb{Z} }$ is an invariant measure.) It is
also conservative and ergodic by Theorem 4.5.3 in
\cite{aaronson:1997}. Therefore, the group action $\mathcal{T} = \{
T^\bn: \bn \in \bbz^d \}$ is conservative, ergodic and measure
preserving on   $(E, \mathcal{E}, \mu)$. 

Equipped with a measure preserving group action on the space
$(E,\mathcal{E})$ we can now  define a stationary symmetric
$\alpha$-stable random field by 
\begin{equation} \label{eq; SaS d-para random fields}
X_\bn = \int_E f \circ T^\bn (\bx)\, M(d\bx),   \quad \bn \in\bbz^ d\,,
\end{equation}
where  $M$ is a $S \alpha S$ random measure on $(E,\mathcal{E})$ with
control measure $\mu$, and 
\begin{equation} \label{e:kernel}
  f (\bx) =  \one( \bx^{(i)}\in A, \, i=1,\ldots, d), \ \bx=(\bx^{(1)},\ldots,
\bx^{(d)})\,.
\end{equation} 
where $A=\{ \bx\in\bbz^\bbz:\, x_0=0\}$. Clearly, $f\in  L^\alpha (
\mu)$, which guarantees that the integral in \eqref{eq; SaS d-para
  random fields} is well defined. 
We refer the reader to  \cite{samorodnitsky:taqqu:1994} for general
information on stable processes and integrals with respect to stable
measures, and to \cite{rosinski:2000} on more details on  stationary
stable random fields and their representations.

The random field model defined by \eqref{eq; SaS d-para
  random fields} is attractive because the key parameters involve
in its definition have a clear intuitive meaning: the index of
stability $0<\alpha<2$ is responsible for the heaviness of the tails,
while $0<\beta_i<1$, $i=1,\ldots, d$ (defined in
\eqref{e:return.regvar}) are responsible for the ``length of the
memory''. The latter claim is not immediately obvious, but its
(informal) validity will become clearer in the sequel. 

The following array of positive numbers will play the crucial role in
the extremal limit theorems in this paper. Denote for $n=1,2,\ldots$
and $i=1,\ldots, d$, 
$$
b_{n}^{(i)}= \bigl( \mu_i \left( \{ \bx:\, x_k =
  0 \; \text{for some} \; k=0,1,\ldots,n \}\right)\bigr)^{1/\alpha}\,,
$$
and let 
\begin{equation} \label{e:b.n}
b_\bn = \prod_{i=1}^d b_{n_i}^{(i)}, \ \bn=(n_1,\ldots,
n_d)\in\bbn_0^d\,. 
\end{equation}
Then $b_\bn^\alpha=\mu(B_\bn)$, where 
$$
B_\bn=    \{  \bx= (\bx^{(1)},\ldots,
\bx^{(d)})  \in E :  x^{(i)}_{k_i} = 0 \ 
  \text{ for some } \ 0\leq k_i \leq  n_i, \ \text{each} \ i=1,\ldots,
  d\}\,. 
$$
Therefore, we can define, for each $\bn \in\bbn_0^d$, a probability measure
$\eta_\bn$  on $(E, \mathcal E)$ by 
\begin{equation} \label{e:eta.n}
\eta_\bn(\cdot) = b_\bn^{-\alpha}\mu   (  \cdot  \;   \cap  B_\bn)\,.
\end{equation}
This probability measure allows us to represent the restriction of the
stationary $S\alpha S$ random field 
$\mathbf{X}$ in (\ref{eq; SaS d-para random fields}) to the hypercube
$[\0,\bn]=\{\mathbf{0} \leq \bk \leq  \mathbf{n}\}$ as a series,
described below, and that we will find useful in the sequel. 
It is useful to note also that the measure $\eta_\bn$ is the product
measure of $d$ probability measures on $
\bigl(\mathbb{Z}^{\bbz} , \mathcal{B}(\mathbb{Z}^{\bbz}) \bigr)$:
$\eta_\bn = \eta_{n_1}^{(1)}\times \cdot\times \eta_{n_d}^{(d)}$ for
$\bn=(n_1,\ldots, n_d) \in\bbn_0^d$, where for $i=1,\ldots, d$ and
$n\geq 0$, 
\begin{equation} \label{e:eta.i}
\eta^{(i)}_n(\cdot) = (b_{n}^{(i)})^{-\alpha} \mu_i   \bigl(  \cdot  \;   \cap \{
\bx\in \mathbb{Z}^{\bbz}:\, x_k=0  \ \text{for some} \ 0\leq k\leq
n\}\bigr)\,. 
\end{equation} 

The restriction of the stationary $S\alpha S$ random field
$\mathbf{X}$ in (\ref{eq; SaS d-para random fields}) to the hypercube
 $[\0,\bn]$ admits, in law, the series representation 
\begin{equation} \label{eq; series rep}
X_\bk    =    b_\bn C_\alpha ^ { 1 / \alpha } \sum _{ j=1} ^ \infty \epsilon_j \Gamma_j ^ { - 1 / \alpha } 
\one_{A^d} \circ T^\bk(U_{j,\bn}), \   \mathbf{0} \leq \bk \leq \mathbf{n}\,,
\end{equation}
with $A^d = A\times \cdot \times A$ the direct product of d copies of A and A  is in \eqref{e:kernel}, where 
 the constant $C_\alpha$ is the tail constant of the $\alpha$-stable
 random variable:
$$
C_\alpha = \left( \int_0^ \infty x ^ { - \alpha } \sin x d x
\right)^{-1}  = \begin{cases}  \frac{1 - \alpha}{ \Gamma(2 - \alpha )
    \cos (\pi \alpha / 2) } \quad & \alpha \neq 1  \\  2/ \pi   &
  \alpha =1 \end{cases} \,.
$$
Furthermore, $\{ \epsilon_j \}$ is a iid sequence of Rademacher random
variables, $\{ \Gamma_j \}$ is the sequence of the arrival times of a
unit rate Poisson process on $(0 , \infty) $, and 
 $\{ U_{j,\bn}\}$  are iid $E$-valued random elements with
 common law $\eta_\bn$. The sequence  
    $\{ \epsilon_j \}$, $\{ \Gamma_j \}$ and $\{ U_{j,\bn}\}$ are
    independent. See \cite{samorodnitsky:taqqu:1994} for details.

\section{Stable regenerative sets and random sup
  measures}  \label{sec:supmeasures} 

In this section we describe the limiting object one obtains in an
extremal limit theorem from the random field $\BX$ of the previous
section. We start with a bit of technical background information on
random closed sets and random sup measures. The reader should consult
\cite{molchanov:2017} for more details. 

Let $\mathbb{E}$ be a locally compact and second countable  Hausdorff
topological space (it will be  $\bbr^d$ or $[0,1]^d$ in our case). We
denote by  $\mathcal{G},\mathcal{F}, \mathcal{K}$ the families of
open, closed, compact sets of $\mathbb{E}$, respectively. 
The Fell topology  on the space $\mathcal{F} $ of closed sets has a
subbasis consisting of the sets 
\begin{align*}
& \mathcal{F}_G = \{ F \in \mathcal{F} : F \cap G \neq \emptyset \}, \
  G \in \mathcal{G}  \\
&  \mathcal{F} ^ K  = \{ F \in \mathcal{F} : F \cap K = \emptyset \},
  \ K \in \mathcal{K}\,.
\end{align*}
The Fell topology is metrizable and compact. 

A random closed set is a measurable mapping from a probability space
to $\mathcal{F}$ equipped with the Borel $\sigma$-field
$\mathcal{B}(\mathcal{F})$ generated by the Fell topology. A specific
random closed set in $\bbr$, the so-called stable regenerative set, is
the key for describing the main results of this paper. 
 
For  $0< \beta < 1$ let $(L_\beta(t),\, t\geq 0)$ be 
the standard $\beta$-stable subordinator. That is, it is 
an increasing L\'{e}vy process with Laplace transform $\mathbb{E} e ^
{ -\theta L_\beta (t) }  = e ^ { - t \theta ^ {  \beta } } $,
$\theta\geq 0$. 
The $\beta$-stable regenerative set is defined to be 
the closure of the range of the 
$\beta$-subordinator, viewed as a random closed set of $\mathbb{R}$: 
\begin{equation}   \label{eq; 1-b stable regenerative sets}
R_{\beta} \stackrel{d}{:=} \overline{ \{ L_{\beta}(t)\,,
  t \geq 0  \} }\,.
\end{equation}
See e.g. \cite{fitzsimmons:taksar:1988}. Products of shifted stable
regenerative sets produce random closed subsets of $\bbr^d$ as
follows. 

For $0<\beta_i<1$, $i=1,\ldots, d$, let 
$R^{(i)} _{\beta_i}, \, i=1,\ldots,d$
be  independent  $\beta_i$-stable regenerative sets.  Let $v^{(i)}>0,
\, i=1,\ldots, d$, and denote $\tilde{R}^{(i)}_{\beta_i}=
v^{(i)}+{R}^{(i)}_{\beta_i}$. Then 
\begin{align}
\tilde{R}_{\bbeta}: = \prod_{i=1}^ d
  \tilde{R}^{(i)}_{\beta_i}        \label{eq; prod shifted sub} 
\end{align}
is a random closed subset of $\mathbb{R}^d$. Such random closed sets
have interesting intersection properties.   The following proposition follows
from Lemma 3.1 of \cite{samorodnitsky:wang:2017}.  
\begin{proposition}  \label{prop;intersection properties of prod shifted rsc}
Let $\{\tilde{R}_{\bbeta,j} \}_{ j \geq 1 } $ be  independent random
closed sets in $\mathbb{R}^d$ as defined by \eqref{eq; prod shifted
  sub}. Suppose that the corresponding shift vectors $(v^{(i)}_j, \,
i=1,\ldots,d)_{j\geq 1}$ satisfy $v^{(i)}_{j_1}\not=v^{(i)}_{j_2}$ if
$j_1\not=j_2$ for each $i=1,\ldots, d$.  Then for any 
$m=1,2,\ldots$, 
$$
P( \cap_{j=1}^m \tilde{R}_{\bbeta, j } \neq \emptyset )= 0 \ \text{or}
\ 1\,.
$$
 The
probability is equal to 1 if and only if $ m <   \min_{i=1,\ldots,
  d}(1- \beta_i)^{-1} $.   
\end{proposition}

The next object to define is a sup measure. For simplicity we take 
$\mathbb{E}$ be the space $[0,1]^d$ or $\mathbb{R}^d$. The details of
the presentation below can be found in \cite{obrien:torfs:vervaat:1990}. 
A  map $m: \mathcal{G} \rightarrow [0, \infty]$ a called sup measure if
$m ( \emptyset) = 0$ and for an arbitrary collection of open sets $\{
G_\gamma  \}$ we have 
 $m ( \cup_{\gamma} G_\gamma ) = \sup_\gamma m(G_\gamma)$.  
 
The sup derivative $d^\vee m$ of a sup measure $m$ is defined by 
\begin{equation} \label{e:sup.der}
d ^ \vee m (t) := \inf_{t \in G } m (G), \quad G \in \mathcal{G}\,. 
\end{equation}
 It is automatically an upper semi-continuous function. 
 Conversely,  for any function $f: \mathbb{E}
 \rightarrow   [0,\infty]$, its  sup integral $i ^ \vee f $ is defined as
 \begin{equation} \label{e:sup.int}
 i^ \vee f (G) := \sup_ { t \in G } f (t) , \quad G \in
 \mathcal{G}\,. 
\end{equation}
If $f$ is upper semi-continuous, then $f = d ^ \vee i ^ \vee f
$. Furthermore, $m=i^\vee d^\vee m$, and one can use \eqref{e:sup.int}
to extend the domain of a sup measures to sets that are not
necessarily open, by setting 
$$
m(B) := \sup_{t \in B} d ^ \vee m(t), \quad B \subset   \mathbb{E} \,.
$$

On the space SM of all sup measures we introduce a topology, the
so-called sup vague topology, by saying that a sequence $\{m_n\}$ of
sup measures converges to a sup measure $m$ if 
$$ \limsup_{ n \to \infty} m_n (K)  \leq m (K) \ \text{for all  $K \in
  \mathcal{K}$ \ and} \ \ 
 \liminf_{ n \to \infty } m_n (G)  \geq m(G) \ \text{for all  $G \in
   \mathcal{G}$.}
$$
The space of  sup measures with sup vague topology is compact
and metrizable; see    Theorem 2.4. in   \cite{norberg:1990}, and we
will often use the notation $\mathcal{M}(\mathbb{E})$ for the space of
sup measures on $\mathbb{E}$. 

A random sup measure  is a measurable map from a probability space
into  SM equipped with the Borel $\sigma$-field  induced by the sup
vague topology. For a random  sup measure $\eta$, a continuity set is
an open set $G$ such that $\eta(G)=\eta({\bar G})$ (the closure of
$G$) a.s., and a useful criterion for weak convergence in the sup
vague topology of random sup-measures is as follows. 
Let $\{\eta_n\}_{n \geq 1 }$ be a sequence of random sup measures, and
$\eta$ a random sup measure. Then $\eta_n \Rightarrow \eta$ if and only
if  
\begin{equation} \label{e:sup.m.weak}
(\eta_n(B_1), \ldots, \eta_n(B_m) ) \Rightarrow (\eta(B_1), \ldots, \eta(B_m) )
\end{equation}
for arbitrary  disjoint open rectangles $B_1, \ldots, B_m$ in
$\mathbb{E}$  that are continuity sets for $\eta$. 
 
We are now ready to construct the random sup measure that will appear
as the limit in the extremal limit theorem in the space SM of the next
section. We will define this measure through its sup derivative, which
is a random upper semi-continuous function. Let $0<\beta_i<1$,
$i=1,\ldots, d$. We start with $d$ independent families of iid
$\beta_i$-stable regenerative sets  $\{R^{(i)} _{\beta_i, j} \}_{j
  \geq 1 }$, $i=1,\ldots, d$. Furthermore,  let $(U_{\alpha,j} ,
V_{\bbeta,j })_{j \geq 1}$ be a measurable enumeration of the points
of a Poisson point process on $\mathbb{R}\times \mathbb{R}^ d$,
independent of the stable regenerative sets,  with the mean measure 
$$
\alpha u ^ { - 1 - \alpha } d u   \prod_{i=1}^ d ( 1 - \beta_i)
  v_i ^ { -\beta_i}d v _i, \ \ u,v_1,\ldots, v_d>0\,.
$$
Then the  triples $(U_{\alpha,j} , V_{\bbeta,j }, R_{\bbeta,j} )_{j
  \geq 1}$ form a Poisson
point process on $\mathbb{R} \times \mathbb{R}^ d  \times
\mathcal{F}(\mathbb{R}^d)$ with the mean measure 
\begin{equation}  \label{eq; eta a b unrestricted mean measure}
\alpha u ^ { - 1 - \alpha } d u  \left( \prod_{i=1}^ d ( 1 - \beta_i)
  v_i ^ { -\beta_i}d v _i  \right) d \tilde P_{\bbeta}, \ \
u,v_1,\ldots, v_d>0\,.  
\end{equation}
Here $\tilde P_{\bbeta}$ is a probability measure on
$\mathcal{F}(\mathbb{R})^d$ defined by 
$$
\tilde P_{\bbeta}= \left( P_{\beta_1} \times \cdots \times P_{\beta_d}  \right)  \circ H ^ {-1}  \,,
$$
with $P_{\beta}$ being the law of the $\beta$-stable regenerative set, 
in  (\ref{eq; 1-b stable regenerative sets}), and $H:\,
\left(\mathcal{F}(\mathbb{R})\right)^d \to
\mathcal{F}(\mathbb{R}^d)$ is defined by 
$$
H(F_1,\ldots, F_d)=F_1\times \cdots\times F_d\,.
$$
Let 
\begin{equation}  \label{eq; eta a b sup measure unrestricted}
\eta_{\alpha, \bbeta} (\bt) = \sum_{j=1}^ \infty U_{\alpha,j} \one_{
  \{ \bt \in V_{\bbeta,j }+ R_{\bbeta,j} \} }\, \ \bt\in\bbr^d\,.
\end{equation}

Several observations are in order. First of all, by Proposition
\ref{prop;intersection properties of prod shifted rsc}, on event of
probability 1, for each $\bt$ the series in \eqref{eq; eta a b sup
  measure unrestricted} has less than 
$$
\ell(\bbeta):= \min_{i=1,\ldots,  d}(1- \beta_i)^{-1}
$$
non-zero terms, so there are no convergence issues. On the same event 
the function defined by \eqref{eq; eta a b sup
  measure unrestricted} is upper semi-continuous. Indeed, for any
finite $\ell$ the function 
$$
\sum_{j=1}^ \ell U_{\alpha,j} \one_{
  \{ \bt \in V_{\bbeta,j }+ R_{\bbeta,j} \} }\, \ \bt\in\bbr^d
$$
is upper semi-continuous since each terms in this finite sum is 
 upper semi-continuous due to the fact that each shifted product of
 stable regenerative sets is a closed set. Moreover, it is easy to
 check that, on each compact set, the uniform distance between this
 function and that defined in \eqref{eq; eta a b sup
  measure unrestricted}, goes to zero as $ \ell \to\infty$; see p. 10 in
\cite{samorodnitsky:wang:2017}. 

We now define a random sup measure as the sup integral of the random 
upper semi-continuous function in \eqref{eq; eta a b sup
  measure unrestricted}, and we will use the same notation,
$\eta_{\alpha, \bbeta}$, for this sup measure. That is,
\begin{equation}   \label{eq; eta a b sup measure}
\eta_{\alpha, \bbeta} (B) = \sup_{ t \in B } \sum_{j=1}^ \infty U_{\alpha,j} \one_{
  \{ \bt \in V_{\bbeta,j }+ R_{\bbeta,j} \} },\, \  B\in 
\mathcal{B}(\bbr^d)\,.
\end{equation}

\begin{remark} \label{rk:prop.eta}
{\rm 
The random sup measure $\eta_{\alpha,\bbeta}$ defined by \eqref{eq;
  eta a b sup measure} 
 is stationary, in the sense that for every
$\bx\geq\0$, $\eta_{\alpha,\bbeta}(\cdot+\bx)\eid
\eta_{\alpha,\bbeta}$. This follows from the shift invariance of the
law the random 
upper semi-continuous function in \eqref{eq; eta a b sup
  measure unrestricted} as in Proposition
3.2 in \cite{samorodnitsky:wang:2017} dealing with the 
  case $d=1$. The argument in that proposition also shows that 
the random sup measure $\eta_{\alpha,\bbeta}$ is  
self-similar, in the sense that for any $c_1>0, \ldots, c_d>0$, 
$$
\eta_{\alpha,\bbeta}\circ p_{c_1,\ldots, c_d}\eid \prod_{i=1}^d
c_i^{(1-\beta_i)/ \alpha} \eta_{\alpha,\bbeta}\,,
$$
where $ p_{c_1,\ldots, c_d}:\, \bbr^d\to \bbr^d$ is the
  multiplication functional $p_{c_1,\ldots, c_d}(t_1,\ldots, t_d) =
  (c_1t_1,\ldots, c_dt_d)$.   

Importantly,, $\eta_{\alpha,\bbeta}$ is a Fr\'echet
  random sup measure if and only if the sets $(V_{\bbeta,j }+
  R_{\bbeta,j}),\, j=1,2,\ldots$ are a.s. disjoint. According to
  Proposition \ref{prop;intersection properties of prod shifted rsc},
  a necessary and sufficient condition for this is $\beta_i\leq 1/2$
  for some $i=1,\ldots, d$. }
\end{remark}

The restriction of the random sup measure  $\eta_{\alpha,\bbeta}$ in
\eqref{eq; eta a b sup measure}   to the hypercube
$[\0,\one]$ has a somewhat more convenient representation. 
Let $\{R^{(i)} _{\beta_i, j} \}_{j  \geq 1 }$, $i=1,\ldots, d$ be as
stable regenerative sets as above, and let  $\{V^{(i)}_j \}_{j \geq 1}
$ be $d$ independent 
families of iid random variables on $[0,1]$ with distributions 
given by 
\begin{equation}  \label{eq; shift }
P( V^{(i)}_1 \leq x ) := x ^ {1-\beta_i}, \quad  x \in [0,1]\,.
\end{equation}
Let now  $\{ \Gamma_j \}$ be the sequence of the arrival times of a
unit rate Poisson process on $(0 , \infty) $. 
Assume that the families $\{V^{(1)}_j \}_{j \geq 1} ,\ldots, \{V^{(d)}_j \}_{j \geq 1}
$, $\{R^{(1)} _{\beta_1, j} \}_{j \geq 1 }, \ldots, \{R^{(d)}
_{\beta_d, j} \}_{j \geq 1 } $ and the Poisson process are
independent.  Denoting
$$
\tilde{R}^{(i)}_{\beta_i, j } = V^{(i)}_j + R^{(i)}_{\beta_i, j }, \
\ 
 1 \leq i \leq d,  \ j \geq 1  
$$
and
$$
  \tilde{R}_{\bbeta, j } = \prod_{i=1}^ d
  \tilde{R}^{(i)}_{\beta_i, j } \subset \mathbb{R}^d\,,        
$$
an alternative representation for the  random 
upper semi-continuous function in \eqref{eq; eta a b sup
  measure unrestricted} restricted to $[\0,\one]$ is   
\begin{equation} \label{eq; eta a b sup derivative}
\eta_{\alpha, \bbeta}(\bt)=  \sum_{j =1 }^ \infty  \Gamma_j ^ {-1 /
  \alpha } \one_{ \{ \bt \in \tilde{R}_{\bbeta,j}  \} }, \quad 
\bt \in [0,1]^d\,,
\end{equation}
with the corresponding change in \eqref{eq; eta a b sup measure}.  
 
\section{Convergence of the random sup measures}  \label{Convergence
  of the random sup measures}

In this section we establish the first functional extremal theorem for
the  stationary random field $\mathbf{X}$ in (\ref{eq; SaS d-para
  random fields}). The  random field   naturally induces a family of
random sup-measures  
$\{\eta_\bn \}_{\bn \in \bbn^d}$  by 
\begin{equation}  \label{e:M.n}
\eta_\bn(B) : = \max_{\bk/\bn \in B } X_\bk  , \quad B \in
\mathcal{B}( [0, \infty)^d ) \,.
\end{equation}
In the  following theorem we prove an extremal theorem in the space of
the random sup measures.
\begin{theorem}  \label{thm; r s-m conv}
For all $0< \alpha < 2 $ and $0 < \beta_i < 1 $, $i=1,\ldots,d$, 
\begin{equation} \label{e:Mn.conv}
\frac{1}{b_\bn} \eta_\bn \Rightarrow \left( \frac{ C_\alpha }{2}
\right )^{1/ \alpha }  \eta_{\alpha , \bbeta}, \quad \bn \to \binfty\,,
\end{equation}
where $\eta_{\alpha , \bbeta}$ is the random sup-measure defined in
\eqref{eq; eta a b sup measure unrestricted}. 
The weak convergence holds in the space of sup measures
$\mathcal{M}(\mathbb{R}^d)$ equipped with the sup vague topology. 
\end{theorem}

To simplify the notation, we will  show
the weak convergence in  $\mathcal{M}([\0,\one])$. Note that by
\eqref{eq; series rep} we can represent, in law, the sup measure in
the left hand side of \eqref{e:Mn.conv} as 
\begin{equation} \label{e:Mn.series}
\frac{1}{b_\bn} \eta_\bn(B) = \max_{\bk/\bn \in  B } C_\alpha ^ { 1 /
  \alpha } \sum _{ j=1} ^ \infty \epsilon_j \Gamma_j ^ { - 1 / \alpha
}  
\one_{A^d} \circ T^\bk(U_{j,\bn}), \    B\in {\mathcal B}([0,1]^d)\,. 
\end{equation}

As it is often done, we prove 
Theorem \ref{thm; r s-m conv} via a truncation argument.  
We fix an  $\ell \in \mathbb{N}$ and construct a truncated
random sup-measure $\eta_{\bn,\ell}$ so that 
\begin{equation} \label{e:Mn.trunc}
\frac{1}{b_\bn} \eta_{\bn,\ell}(B) = \max_{\bk/\bn \in  B } C_\alpha ^ { 1 /
  \alpha } \sum _{ j=1} ^ \ell \epsilon_j \Gamma_j ^ { - 1 / \alpha
}  
\one_{A^d} \circ T^\bk(U_{j,\bn}), \    B\in {\mathcal B}([0,1]^d)\,. 
\end{equation}

Note that we can write 
$$
U_{j,\bn}(\bk) = \bigl( U_{j,n_1}^{(1)}(k_1),\ldots, U_{j,n_d}^{(d)}(k_d)\bigr)
$$
for $\bn=(n_1,\ldots, n_d)$ and $\bk=(k_1,\ldots, k_d)$, with
independent components in the right hand side, where $U_{j,n}^{(i)}$
has the law $\eta^{(i)}$ given in \eqref{e:eta.i}, $i=1,\ldots, d$. 
Therefore, the set of zeroes of $U_{j,\bn}$ satisfies 
$$ \mathcal{Z}(U_{j,\bn}) = \mathcal{Z}(U^{(1)}_{j,n_1} ) \times
\cdots \times \mathcal{Z}( U ^ { (d)}_{j,n_d} )\,. $$

To proceed, we need to introduce new notation. Let $S\subset
\mathbb{N}$. We set 
\begin{align*}
 &\hat{I}^{(i)}_{S,n}= \cap_{j \in S} \mathcal{Z}(U^{(i)}_{j,n}), \,
   i=1,\ldots, d, \, n\geq 1, \ 
\hat{I}_{S,\bn}  = \cap_{j \in S} \mathcal{Z}(U_{j,\bn}), \, \bn\in
  \bbn^d, \\
& I^{(i)}_S  = \cap_{j \in S}  \tilde{R}^{(i)}_{\beta_i, j }, \,
  i=1,\ldots, d, \    I_S  = \cap_{j \in S} \tilde{R}_{\bbeta,j} \,.
\end{align*}
At this stage the random objects described above do not need to be
defined on the same probability space. We need the following extension
of Theorem 5.4 of \cite{samorodnitsky:wang:2017}.

\begin{proposition}   \label{lem; joint convergence lemma}
$$ 
\left ( \frac1 \bn \hat{I}_{S,\bn}   \right )_{S  \subset \{ 1 , \ldots, \ell \}
}  \Rightarrow  (I_S)_{ S  \subset \{ 1 , \ldots, \ell \} }, \ \bn\to\binfty\,,
$$ 
in $\bigl(\mathcal{F}( [\0,\one] )\bigr)^{2^\ell}$.
\end{proposition}
\begin{proof}
By Theorem 5.4 of \cite{samorodnitsky:wang:2017}, for each
$i=1,\ldots, d$ and $S\subset \{1,\ldots, \ell\}$, 
$$
\frac{1}{n}  \left[ \hat{I}^{(i)}_{S,n}  \cap [0,1]\right]\Rightarrow
I^{(i)}_S \cap [0,1] , \quad n \to \infty\,,
$$ 
in the sense of weak convergence of random closed sets. By Corollaries
1.7.13 and 1.7.14  in \cite{molchanov:2017} applied to rectangles of
the type $\prod_{i=1}^d [a_i,b_i]$, $0\leq a_i\leq b_i\leq 1, \,
i=1,\ldots, d$, we conclude that for every $S\subset \{1,\ldots,
\ell\}$,
$$
 \frac{1}{\bn} \hat{I}_{S,\bn}  \Rightarrow  I_S , \ \bn\to\binfty\,,
$$ 
 $(\mathcal{F}( [\0,\one] )$. By Theorem 2.1 (ii) in
\cite{samorodnitsky:wang:2017}, this implies the joint convergence in
the proposition.  
\end{proof}

For $S\subset \{ 1, \ldots, \ell \}$ we define now 
\begin{equation}  \label{eq; exclusive return time I star S}
\hat{I}_{S,\bn} ^ \ast  =  \hat{I}_{S,n}  \cap  \left(  \bigcup_{ j \in
    \{1,\ldots, \ell \} \setminus S }  \mathcal{Z}(U_{j,\bn} ) \right )^c\,,
\end{equation}
the set of times where only the Markov chains corresponding to $j\in
S$ reach $0$.  Similarly we define
\begin{equation}  \label{eq; exclusive range I star S}
I_S^ \ast  =  I_S  \cap  \left(  \bigcup_{ j \in
    \{1,\ldots, \ell \} \setminus S }  \tilde{R}_{\bbeta,j} \right )^c\,,
\end{equation}
As in the case of the one-dimensional time, for large $\bn$ the sets 
$\hat{I}_{S,\bn} ^ \ast$ and $I_{S} ^ \ast$ are likely to be
alike. 

\begin{lemma}   \label{lem; asy exclusive return time}
For an open rectangle $B \subset [0,1]^d$, let $H_\bn(B)$ be the event 
\begin{equation}  \label{eq; HnB vanishing event}
H_\bn (B) : =   \bigcup_{ S \subset \{1, \ldots, \ell \} } \left(
  \left \{ \frac{\one}{\bn} \hat{I}_{S,\bn} \cap B \neq \emptyset  \right \}  \cap
  \left \{ \frac{\one}{\bn}\hat{I}^\ast_{S,\bn}  \cap B = \emptyset \right \}   \right )  
\end{equation}
Then, $ \lim_{\bn \to \binfty } P ( H_\bn(B) ) = 0 $. 
\end{lemma}
\begin{proof}
Write $B=B_1 \times \cdots \times B_d$, with $B_1, \ldots,  B_d$ open
rectangles in $[0,1]$.  Denoting
$$
\hat{I}^{(i)\ast}_{S,n} := \hat{I}^{(i)}_{S,n}  \cap  \left(
  \bigcup_{ j \in \{1,\ldots, \ell \} \setminus  S }
  \mathcal{Z}(U^{(i)}_{j,n} ) \right )^c, \ i=1,\ldots, d, \ S\subset
\{1,\ldots, \ell\}, \ n=1,2,\ldots\,,
$$
we have 
\begin{align*}
 &H_\bn(B)   
\subset \bigcup_{ S \subset \{1, \ldots, \ell \} }
   \bigcup_{i=1,\ldots, d}\left(     \left \{ \frac{1}{n_i}\hat{I}^{(i)}_{S,n_i} \cap
   B_i \neq \emptyset  \right \}  \cap \left \{ \frac{1}{n_i}
   \hat{I}^{(i)\ast}_{S,n_i}  \cap B_i = \emptyset \right \}   \right) \,.
\end{align*}
The right hand side above is a finite union events, and the
probability of each one is asymptoticly  vanishing by Lemma 5.5 in
\cite{samorodnitsky:wang:2017}. 
\end{proof}

\begin{remark} \label{rk:alone} {\rm
The argument of Lemma 5.5 in \cite{samorodnitsky:wang:2017} shows also
the following version of the lemma: let 
$$
H_\bn^*= \bigcup_{a_i>0, i=1,\ldots,d} H_\bn\left( \prod_{i=1}^d
  (0,a_i)\right)\,.
$$
Then $ \lim_{\bn \to \binfty } P ( H_\bn^* ) = 0 $. We will find this
formulation useful in the sequel. 
}
\end{remark}

We are now ready to prove convergence of the truncated random
sup-measures.  
\begin{proposition}   \label{prop; truncated random sup-mea conv}
Let $\ell\geq 1$, and define a random sup-measure $\eta_{\alpha, \bbeta,
  \ell}$  by 
\begin{equation}  \label{eq; eta a b l truncated rs-m}
 \eta_{\alpha , \beta, \ell } (B)  =   \sup_{ \bt \in B }
 \sum_{j=1}^ \ell  \Gamma_j ^ { - 1 / \alpha } \one_{ \{ \bt \in
   \tilde{R}_{\beta, j } \} },  \quad B \in \mathcal{B}([0,1]^d) \,. 
 \end{equation}
Then   
$$
\frac{1}{b_\bn} \eta_{ \bn , \ell } \Rightarrow  \left(
  \frac{C_\alpha}{2} \right)^{1 / \alpha }  \eta_{\alpha, \bbeta,
  \ell}, \quad \bn \to \binfty 
$$
in the space of sup measures
$\mathcal{M}([\0,\one])$ equipped with the sup vague topology. 
\end{proposition}
\begin{proof}
We start by observing that an alternative expression for the random
sup-measure $\eta_{\alpha, \bbeta,   \ell}$  is
\begin{equation}  \label{e:eta.l.1}
\eta_{\alpha , \beta, \ell } (B) =\max_{S \subset \{1, \ldots, \ell \}
  } \one_{ \{ I_S \cap B\neq \emptyset \} }  \sum_{j \in S } \Gamma_j^{-1/\alpha} \,.
\end{equation}
Since stable subordinators do not hit fixed points, 
by  \eqref{e:sup.m.weak} 
it suffices to show that for any $m$ disjoint open rectangles
$B_r=\prod_{i=1}^d (a_i^{(r)}, b_i^{(r)}), \, r=1,\ldots, m$ 
in $ [0,1]^d$, we have a convergence of random vectors:
$$
\frac{1}{b_\bn}\left( 
\eta_{\bn, \ell} (B_1), \ldots,\eta_{\bn, \ell} (B_m)  \right)
\Rightarrow \left(\frac{C_\alpha}{2} \right)^{1 / \alpha } \left(
\eta_{\alpha, \bbeta, \ell}(B_1), \ldots,  \eta_{\alpha, \bbeta,
  \ell}(B_m) \right)\,.
$$
It is clear that for any $ r=1,\ldots, m$, on the compliment of the
event $H_\bn(B_r)$,
\begin{equation*}   
 \max_{\bk/\bn \in  B_r}    \sum _{ j=1} ^ \ell \epsilon_j \Gamma_j^{ - 1 / \alpha}  
\one_{A^d} \circ T^\bk(U_{j,\bn}) =
 \max_{S \subset \{1, \ldots, \ell \} } \one_{ \{ (\one/\bn)\hat{I}_{S,\bn} \cap
   B_r \neq \emptyset \} }  \sum_{j \in S }  \epsilon_j  \Gamma_j^{ -
   1 / \alpha}\,. 
 \end{equation*}
Since $\hat{I}_{S,\bn}$ is decreasing as the set $S$ increases, we can
choose, for a fixed $S$, the set 
$S^\prime = \{ j \in S: \epsilon_j = 1 \}$  to obtain 
\begin{equation*}    
  \max_{S \subset \{1, \ldots, \ell \} } \one_{ \{ (\one/\bn)\hat{I}_{S,\bn} \cap
   B_r \neq \emptyset \} }  \sum_{j \in S }  \epsilon_j  \Gamma_j^{ -
   1 / \alpha}
 =  \max_{S \subset \{1, \ldots, \ell \} } \one_{ \{ (\one/\bn)\hat{I}_{S,\bn} \cap
   B_r \neq \emptyset \} }  \sum_{j \in S }  \one(\epsilon_j=1)
 \Gamma_j^{ -  1 / \alpha}  \,.
\end{equation*}
Hence, on the compliment of the event $H_\bn(B_1) \cup \cdots \cup
H_\bn(B_m)$,   
\begin{align*} 
 \frac{1}{b_\bn} \left(   \eta_{\bn,\ell} (B_1), \ldots,
  \eta_{\bn,\ell} (B_\bn)    \right )&  = C_\alpha^{1/ \alpha }
  \left(    \max_{S\subset \{1, \ldots,   \ell \} } 
\one_{ \{ (\one/\bn)\hat{I}_{S,\bn} \cap B_r
   \neq \emptyset \} } \sum_{j \in S }  
\one_{ \{\epsilon_j = 1 \} }  \Gamma_j^{-1/\alpha} \right )_{r=1,\ldots,m} \,. 
\end{align*}
By Proposition \ref{lem; joint convergence lemma} the random vector in
the right hand side converges weakly as $\bn\to\binfty$ to the random
vector
\begin{align*}
 C_\alpha^{1/ \alpha } \left(    \max_{S \subset \{1, \ldots, \ell \}
  } \one_{ \{ I_S \cap B_r \neq \emptyset \} }  \sum_{j \in S }  \one_{
  \{\epsilon_j = 1 \} } \Gamma_j^{-1/\alpha}     \right
  )_{r=1,\ldots,m}\,.
\end{align*}
Since, by Lemma \ref{lem; asy exclusive return time}, the event
$H_\bn(B_1) \cup \cdots \cup H_\bn(B_m)$ has an asymptotically
vanishing probability, the random vector 
$$
\frac{1}{b_\bn} \left(   \eta_{\bn,\ell} (B_1), \ldots,
  \eta_{\bn,\ell} (B_m)    \right )
$$
converges weakly to the same limit. The claim of the proposition
follows by noticing that the thinned Poisson random measure $(\one_{
  \{\epsilon_j = 1 \} } \Gamma_j^{-1/\alpha} )_{ j \geq 1 }$ has the
same law as $(2 ^ { - 1/ \alpha }\Gamma_j^{-1/\alpha} )_{ j
  \geq 1 }$ and using \eqref{e:eta.l.1}. 
\end{proof}

We now deal with the part of the random sup measure in Theorem
\ref{thm; r s-m conv} that is left after the truncation procedure
above. The following proposition is crucial. 
\begin{proposition}   \label{prop; tail sup-mea conv}
For all $\delta>0$,  
\begin{equation}  \label{eq; tail sup-mea conv}
\lim_{\ell \to \infty}  \limsup_{\bn \to \binfty}  P \left(  \max_{
    \mathbf{0} \leq \bk \leq \mathbf{n} }  \left \vert
    \sum_{j=\ell+1}^ \infty    \epsilon_j \Gamma_j ^ { - 1 / \alpha}  
\one_{A^d} \circ T^\bk(U_{j,\bn})
 \right \vert  > \delta   \right) = 0 \,.
\end{equation}
\end{proposition}
\begin{proof}
Clearly,
\begin{align} \label{e:split.l}
& P \left(  \max_{
    \mathbf{0} \leq \bk \leq \mathbf{n} }  \left \vert
    \sum_{j=\ell+1}^ \infty    \epsilon_j \Gamma_j ^ { - 1 / \alpha}  
\one_{A^d} \circ T^\bk(U_{j,\bn})
 \right \vert  > \delta   \right) \\
 \leq &P \left(  \max_{
    \mathbf{0} \leq \bk \leq \mathbf{n} }  \left \vert
    \sum_{j=\ell+1}^ \infty    \epsilon_j \Gamma_j ^ { - 1 / \alpha}
  \one(\Gamma_j>b_\bn^\alpha) \one_{A^d} \circ T^\bk(U_{j,\bn})
 \right \vert  > \delta/2  \right) \notag \\ 
+ &P \left(  \max_{
    \mathbf{0} \leq \bk \leq \mathbf{n} }  \left \vert
    \sum_{j=\ell+1}^ \infty    \epsilon_j \Gamma_j ^ { - 1 / \alpha}  
 \one(\Gamma_j\leq b_\bn^\alpha) \one_{A^d} \circ T^\bk(U_{j,\bn})
 \right \vert  > \delta/2  \right) \,. \notag 
\end{align}
By symmetry,
\begin{align*}
&P \left(  \max_{
    \mathbf{0} \leq \bk \leq \mathbf{n} }  \left \vert
    \sum_{j=\ell+1}^ \infty    \epsilon_j \Gamma_j ^ { - 1 / \alpha}
  \one(\Gamma_j>b_\bn^\alpha) \one_{A^d} \circ T^\bk(U_{j,\bn})
 \right \vert  > \delta/2  \right) \\
\leq 2&P \left(  \max_{
    \mathbf{0} \leq \bk \leq \mathbf{n} }  \left \vert
    \sum_{j=1}^ \infty    \epsilon_j \Gamma_j ^ { - 1 / \alpha}
  \one(\Gamma_j>b_\bn^\alpha) \one_{A^d} \circ T^\bk(U_{j,\bn})
 \right \vert  > \delta/2  \right)\,.
\end{align*}
The sum in the right hand side is a representation, in law, of the
restriction to the set $\0\leq \bk\leq \bn$ of the random field
$(b_\bn^{-1}Y_\bk, \, \bk\in\bbz^d)$, where $(Y_\bk, \, \bk\in\bbz^d)$ is
a 
stationary symmetric infinitely divisible random field defined, similarly
to the original stationary symmetric $\alpha$-stable random field in
\eqref{eq; SaS d-para random fields}, by 
\begin{equation} \label{e:bdd.levy.field}
Y_\bk = \int_E f \circ T^\bk (\bx)\, \tilde M(d\bx),    \quad \bk \in\bbz^
d\,,
\end{equation}
with the distinction that the local L\'evy measure $\tilde\rho$ of the
symmetric infinitely divisible random measure in
\eqref{e:bdd.levy.field} has the density $\alpha |x|^{-(\alpha+1)}$
  restricted to $|x|\leq 1$. See Chapter 3 in
  \cite{samorodnitsky:2016}. In particular, each $Y_\bk$ has a L\'evy
  measure with a bounded support and, hence, has (faster than)
  exponentially fast decaying tails. See e.g. \cite{sato:1999} Chapter 5. We
  conclude by the regular variation \eqref{e:bn.j} of the factors in
  \eqref{e:b.n} that for $\bn=(n_1,\ldots, n_d)$, 
\begin{align*}
&P \left(  \max_{
    \mathbf{0} \leq \bk \leq \mathbf{n} }  \left \vert
    \sum_{j=\ell+1}^ \infty    \epsilon_j \Gamma_j ^ { - 1 / \alpha}
  \one(\Gamma_j>b_\bn^\alpha) \one_{A^d} \circ T^\bk(U_{j,\bn})
 \right \vert  > \delta/2  \right) \\
\leq 2&P\bigl( \max_{
    \mathbf{0} \leq \bk \leq \mathbf{n} }
        |Y_\bk|>b_\bn(\delta/2)\bigr) \leq 2\prod_{i=1}^d (1+n_i)
P\bigl(  |Y_\0|>b_\bn(\delta/2)\bigr) \to 0
\end{align*}
as $\bn\to\binfty$. Therefore, the claim of the proposition will follow
once we prove that 
\begin{equation} \label{e:tail.1}
\lim_{\ell \to \infty}  \limsup_{\bn \to \binfty}  P \left(  \max_{
    \mathbf{0} \leq \bk \leq \mathbf{n} }  \left \vert
    \sum_{j=\ell+1}^ \infty    \epsilon_j \Gamma_j ^ { - 1 / \alpha}  
\one(\Gamma_j\leq b_\bn^\alpha) \one_{A^d} \circ T^\bk(U_{j,\bn})
 \right \vert  > \delta   \right) = 0 \,.
\end{equation}
 
To this end, let $M>0$ and set $D^M_\ell := \{ \Gamma_{\ell + 1 } \geq
M \}$. By the Strong Law of Large Numbers, 
$\lim_{\ell \to \infty} P (D_\ell^ M ) = 1 $. Therefore, we may replace
the probability in \eqref{e:tail.1} by 
$$ 
P \left(  \left \{  \max_{
    \mathbf{0} \leq \bk \leq \mathbf{n} }  \left \vert
    \sum_{j=\ell+1}^ \infty    \epsilon_j \Gamma_j ^ { - 1 / \alpha}  
\one(\Gamma_j\leq b_\bn^\alpha) \one_{A^d} \circ T^\bk(U_{j,\bn})
 \right \vert  > \delta \right \}   \cap  D^M_\ell  \right)\,.
$$
The above quantity does not exceed 
\begin{align}\label{eq; G1 tail estima  0} 
& \sum_{ \mathbf{0}  \leq \bk \leq \mathbf{n} }   P \left( \left\{ \left\vert
    \sum_{j=\ell+1}^ \infty    \epsilon_j \Gamma_j ^ { - 1 / \alpha}  
\one(\Gamma_j\leq b_\bn^\alpha) \one_{A^d} \circ T^\bk(U_{j,\bn})
 \right\vert  > \delta \right\}   \cap  D^M_\ell  \right) \\  
= &  \prod_{i=1}^d (n_i+1)    P \left( \left\{ \left\vert
    \sum_{j=\ell+1}^ \infty    \epsilon_j \Gamma_j ^ { - 1 / \alpha}  
\one(\Gamma_j\leq b_\bn^\alpha) \one_{A^d} (U_{j,\bn})
 \right\vert  > \delta \right\}   \cap  D^M_\ell  \right)\,,  \notag
 \end{align}  
since the probabilities in the sum in \eqref{eq; G1 tail estima  0} do
not depend on $\bk$.  Using symmetry, we have 
\begin{align}
  &   P \left( \left\{ \left\vert
    \sum_{j=\ell+1}^ \infty    \epsilon_j \Gamma_j ^ { - 1 / \alpha}  
\one(\Gamma_j\leq b_\bn^\alpha) \one_{A^d} (U_{j,\bn})
 \right\vert  > \delta \right\}   \cap  D^M_\ell  \right)    \notag  \\ 
\leq &   P \left(   \left\vert
    \sum_{j=\ell+1}^ \infty    \epsilon_j \Gamma_j ^ { - 1 / \alpha}  
\one(M\leq \Gamma_j\leq b_\bn^\alpha) \one_{A^d} (U_{j,\bn})
 \right\vert  > \delta   \right) \notag \\
\leq &   2P \left(   \left\vert
    \sum_{j=1}^ \infty    \epsilon_j \Gamma_j ^ { - 1 / \alpha}  
\one(M\leq \Gamma_j\leq b_\bn^\alpha) \one_{A^d} (U_{j,\bn})
 \right\vert  > \delta   \right)\,.   \label{eq; G1 tail estima 2}   
  \end{align} 
Notice that the Poisson point process $( b_\bn^{-\alpha}  \Gamma_j
\one_{A^d} (U_{j,\bn}))_j$ has the same law as the Poisson point
proces $ ( \Gamma_j)_j$.   Therefore, the probability in  (\ref{eq; G1
  tail estima 2}) coincides with 
\begin{align*}
  & P  \left( b_\bn^{-1}  \left\vert \sum_{j=1}^\infty  \epsilon_j  
\Gamma_j ^{ -1/ \alpha}  \one_{ \{  M b_\bn^{-\alpha} \leq \Gamma_j \leq
  1   \} }\right\vert >\delta\right) 
   \leq   P \left(  \left\vert b_\bn ^ {-1}  \sum_{j=j_M+1}^ \infty \epsilon_j  
\Gamma_j ^{ -1/ \alpha}  \one_{ \{  M b_\bn^{-\alpha} \leq \Gamma_j \leq
  1   \} }\right\vert  > \delta /  2 \right )   \\
   \leq & b_\bn ^ { - p } (\delta/2)^{-p} E \left\vert  \sum_{j=j_M+1}^ \infty \epsilon_j  
\Gamma_j ^{ -1/ \alpha}  \one_{ \{  M b_\bn^{-\alpha} \leq \Gamma_j \leq
  1   \} }\right\vert ^ p    
  \end{align*}
for any $p>0$, where 
$ j_M = \lceil M ^ { 1 /\alpha } \delta/2 \rceil $. If $p>0$ is large
enough, then by the regular variation \eqref{e:bn.j}, 
$$
\lim_{\bn\to\binfty}  \prod_{i=1}^d (n_i+1)  b_\bn ^ { - p }\to 0\,.
$$
Therefore, \eqref{e:tail.1} will follow once we check that for any
$p>0$ we can take $M>0$ large enough so that
\begin{equation} \label{e:finite.mom}
\limsup_{\bn\to\binfty} E \left\vert  \sum_{j=j_M+1}^ \infty \epsilon_j  
\Gamma_j ^{ -1/ \alpha}  \one_{ \{  M
  b_\bn^{-\alpha} \leq \Gamma_j \leq   1   \} }\right\vert ^ p  <\infty\,.
\end{equation}
Let us take $p=2k$, an even integer. 
By the Khintchine inequality (see e.g. (A.1) in \cite{nualart:1995}),
there is a constant $c_p\in (0,\infty)$ such that  
\begin{align*}
&E \left\vert  \sum_{j=j_M+1}^ \infty \epsilon_j  
\Gamma_j ^{ -1/ \alpha}  \one_{ \{  M
  b_\bn^{-\alpha} \leq \Gamma_j \leq   1   \} }\right\vert ^ p
\leq c_p E\left( \sum_{j=j_M+1}^ \infty \bigl( \Gamma_j ^{ -1/ \alpha}  \one_{
  \{  M   b_\bn^{-\alpha} \leq \Gamma_j \leq   1   \}
}\bigr)^2\right)^{p/2} \\
\leq &c_p E\left( \sum_{j=j_M+1}^ \infty \Gamma_j ^{
       -2/\alpha}\right)^{k} 
 \leq c_p\left( \sum_{j=j_M+1}^ \infty
  \bigl( E( \Gamma_j ^{-2k/\alpha})\bigr)^{1/k}\right)^k\,.
\end{align*}
The claim \eqref{e:finite.mom} now follows since $E\bigl( \Gamma_j ^{
       -2k/\alpha}\bigr)<\infty$ for $j>2k/\alpha$, and 
$$
E\bigl( \Gamma_j^{-2k/\alpha}\bigr)\sim j^{-2k/\alpha}, \ j\to\infty\,.
$$ 
\end{proof}

We are now ready to finish the proof of Theorem \ref {thm; r s-m
  conv}.  
\begin{proof} [Proof of Theorem \ref {thm; r s-m conv} ]
Once again, by  \eqref{e:sup.m.weak} 
it suffices to show that for any $m$ disjoint open rectangles
$B_r=\prod_{i=1}^d (a_i^{(r)}, b_i^{(r)}), \, r=1,\ldots, m$ 
in $ [0,1]^d$, we have
$$
\frac{1}{b_\bn}\left( 
\eta_{\bn} (B_1), \ldots,\eta_{\bn, \ell} (B_m)  \right)
\Rightarrow \left(\frac{C_\alpha}{2} \right)^{1 / \alpha } \left(
\eta_{\alpha, \bbeta}(B_1), \ldots,  \eta_{\alpha, \bbeta,
  \ell}(B_m) \right)\,.
$$
Since by  Proposition \ref{prop; truncated random sup-mea conv}
 $$
\frac{1}{b_\bn}\left( 
\eta_{\bn, \ell} (B_1), \ldots,\eta_{\bn, \ell} (B_m)  \right)
\Rightarrow \left(\frac{C_\alpha}{2} \right)^{1 / \alpha } \left(
\eta_{\alpha, \bbeta, \ell}(B_1), \ldots,  \eta_{\alpha, \bbeta,
  \ell}(B_m) \right)\,, 
$$
and $\eta_{\alpha, \bbeta, \ell}$ increases to $\eta_{\alpha, \bbeta}$
almost surely, we can use the ``convergence together'' argument, as in
Theorem 3.2 of \cite{billingsley:1999}. To this end we need to check
that  for each $r=1, \ldots, m$ and any $\epsilon > 0 $, 
$$
\lim_{\ell \to \infty } \limsup_{\bn \to \binfty } P \left (
  \frac{1}{b_\bn}\left \vert  \eta_{\bn}(B_r) -
    \eta_{\bn,\ell}(B_r)   \right \vert > \epsilon  \right ) =
0\,.  
$$ 
This is, however, an immediate conclusion from Proposition \ref{prop;
  tail sup-mea conv}. 
\end{proof}
\begin{remark} \label{rk:not.frechet} {\rm
Note that the limiting sup measure in Theorem \ref{thm; r s-m conv} is
Fr\'echet only when $\beta_i\leq 1/2$   for some $i=1,\ldots, d$. See
Remark \ref{rk:prop.eta}. 
}
\end{remark}

The stationary random field $\mathbf{X}$ in (\ref{eq; SaS d-para
  random fields}) induces another family of random sup measures
$\{\tilde {\eta}_\bn \}_{\bn \in \bbn^d}$  via 
\begin{equation}  \label{e:Mt.n}
\tilde {\eta}_\bn(B) : = \max_{\bk/\bn \in B } |X_\bk|  , \quad B \in
\mathcal{B}( [0, \infty)^d ) \,.
\end{equation}
This family of random sup-measures satisfies the following analogue of
Theorem \ref{thm; r s-m conv}. 
\begin{theorem}  \label{thm; r s-m abs conv}
For all $0< \alpha < 2 $ and $0 < \beta_i < 1 $, $i=1,\ldots,d$, 
\begin{equation} \label{e:Mn.conv.abs}
\frac{1}{b_\bn} \tilde{\eta}_\bn \Rightarrow
\left(\frac{C_\alpha}{2} \right)^{1/\alpha}
\max\bigl(\eta_{\alpha , \bbeta}^{(1)},\,  \eta_{\alpha ,
  \bbeta}^{(2)}\bigr), 
\quad \bn \to \binfty\,, 
\end{equation}
where $\eta_{\alpha , \bbeta}^{(1)}$ and $\eta_{\alpha , \bbeta}^{(2)}$
are two independent copies of the random sup measure defined in
\eqref{eq; eta a b sup measure unrestricted}. 
The weak convergence holds in the space of sup measures
$\mathcal{M}(\mathbb{R}^d)$ equipped with the sup vague topology. 
\end{theorem}
\begin{proof}
Once again, we will show the weak convergence in
$\mathcal{M}([\0,\one])$. We continue using the notation of the proof of
Theorem \ref{thm; r s-m conv}. The same argument as in the that
proof works once we show that, in the obvious notation, for any
$\ell=1,2,\ldots$, 
\begin{equation} \label{eq; truncated abs conv}
\frac{1}{b_\bn}\tilde{\eta}_{\bn,\ell} \Rightarrow  
\frac{C_\alpha^{1/ \alpha }}{2} 
\max\bigl(\eta_{\alpha , \bbeta,\ell}^{(1)},\,  \eta_{\alpha ,
  \bbeta,\ell}^{(2)}\bigr) , \quad \bn \to \binfty\,,   
\end{equation}
where
\begin{equation}  \label{eq; truncated abs rs-mea}
\frac{1}{b_\bn} \tilde{\eta}_{\bn,\ell}(B) = \max_{\bk/\bn \in B }  
\left\vert\sum _{ j=1} ^ \ell \epsilon_j \Gamma_j ^ { - 1 /\alpha}   
\one_{A^d} \circ T^\bk(U_{j,\bn})\right\vert, \    B\in {\mathcal B}([0,1]^d)\,.  
\end{equation}

Outside the vanishing event $H_n(B_1) \cup \cdots \cup H_n(B_m)$ we
now have 
\begin{align*}
\frac{1}{b_\bn} \left(   \tilde{\eta}_{\bn,\ell} (B_1), \ldots,
 \tilde{\eta}_{\bn,\ell} (B_\bn)    \right ) 
=&\,  C_\alpha^{1/ \alpha }
  \left(  \max\left[  \max_{S\subset \{1, \ldots,   \ell \} } 
\one_{ \{ (\one/\bn)\hat{I}_{S,\bn} \cap B_r
   \neq \emptyset \} } \sum_{j \in S }  
\one_{ \{\epsilon_j = 1 \} }  \Gamma_j^{-1/\alpha},\right.\right. \\
&\left.\left. \max_{S\subset \{1, \ldots,   \ell \} } 
\one_{ \{ (\one/\bn)\hat{I}_{S,\bn} \cap B_r
   \neq \emptyset \} } \sum_{j \in S }  
\one_{ \{\epsilon_j = -1 \} }  \Gamma_j^{-1/\alpha}\right]\right)_{r=1,\ldots,m} \,. 
\end{align*}
Using again Proposition \ref{lem; joint convergence lemma} and  Lemma
\ref{lem; asy exclusive return time} we conclude that, as
$\bn\to\binfty$, 
\begin{align*}
\frac{1}{b_\bn} \left(   \tilde{\eta}_{\bn,\ell} (B_1), \ldots,
 \tilde{\eta}_{\bn,\ell} (B_\bn)    \right ) \Rightarrow
&\,  C_\alpha^{1/ \alpha }
  \left(  \max\left[  \max_{S\subset \{1, \ldots,   \ell \} } 
\one_{ \{  {I}_{S} \cap B_r
   \neq \emptyset \} } \sum_{j \in S }  
\one_{ \{\epsilon_j = 1 \} }  \Gamma_j^{-1/\alpha},\right.\right. \\
&\left.\left. \max_{S\subset \{1, \ldots,   \ell \} } 
\one_{ \{  {I}_{S} \cap B_r
   \neq \emptyset \} } \sum_{j \in S }  
\one_{ \{\epsilon_j = -1 \} }  \Gamma_j^{-1/\alpha}\right]\right)_{r=1,\ldots,m} \,. 
\end{align*}
The statement of the theorem now follows since  $(\one_{ \{ \epsilon_j=1 \} } \Gamma_j^{-1/\alpha})_j$ and $(\one_{
  \{ \epsilon_j=-1 \} } \Gamma_j^{-1/\alpha})_j$ are two independent Poisson
random measures, each with the same law as $(2 ^ { - 1/ \alpha
}\Gamma_j^{-1/\alpha} )_{ j \geq 1 }$, and using \eqref{e:eta.l.1}.
\end{proof}

\begin{remark} \label{rk:not.Fr}  {\rm 
It is interesting to observe that in the case $0<\beta_i\leq 1/2$ for
some $i=1,\ldots, d$, the sup measure $\eta_{\alpha,\bbeta}$ is a
Fr\'echet random sup measure, and so \eqref{e:Mn.conv} can be
reformulated as 
\begin{equation} \label{e:Mn.conv1}
\frac{1}{b_\bn} \tilde{\eta}_\bn \Rightarrow
C_\alpha^{1/ \alpha}  \eta_{\alpha , \bbeta}, 
\quad \bn \to \binfty\,. 
\end{equation}
See Remark \ref{rk:prop.eta}. However, if $0<\beta_i> 1/2$ for
all $i=1,\ldots, d$, then the random sup measure
$\eta_{\alpha,\bbeta}$ is not max-stable, so \eqref{e:Mn.conv1} is no
longer a valid statement of Theorem \ref{thm; r s-m abs conv}. 
}
\end{remark}
 
\section{Convergence of the partial maxima
  processes}  \label{Convergence of the partial maxima processes}

In this section we prove another version of a functional extremal
theorem for the  stationary random field $\mathbf{X}$ in (\ref{eq; SaS
  d-para   random fields}). This time we will be working in the space
$D(\mathbb{R}_+^d)$, and the limit will itself be a random field. 
The random field $\BX$ induces an array of partial maxima random fields
$\{ M_\bn \}$ by 
\begin{equation*}
M_{\bn}(\bt) : =  \max_{ \mathbf{0} \leq \bk \leq \bn\bt } X_\bk,
\quad \bt \in \mathbb{R}_+^d\,.
\end{equation*}
The random sup measure $\eta_{\alpha, \bbeta}$ in (\ref{eq; eta a b
  sup measure unrestricted}) also induces a random field $W_{\alpha,
  \bbeta}$, by 
\begin{equation*}
W_{\alpha, \bbeta} (\bt) : = \eta_{\alpha, \bbeta}( [\0, \bt ]), \quad
\bt \in \mathbb{R}_+^d\,.
\end{equation*}
\begin{remark}\label{rk:ss.field}{\rm
It follows immediately from Remark \ref{rk:prop.eta} that the random
field $\bigl( W_{\alpha, \bbeta} (\bt), \, \bt\in \mathbb{R}_+^d\bigr)$ is 
self-similar, in the sense that for any $c_1>0, \ldots, c_d>0$
$$
\bigl( W_{\alpha, \bbeta} ((c_1t_1,\ldots, c_dt_d)), \, \bt\in
\mathbb{R}_+^d\bigr) \eid \left(
\prod_{i=1}^d c_i^{(1-\beta_i)/ \alpha}  W_{\alpha, \bbeta} (\bt), \,
\bt\in \mathbb{R}_+^d \right)\,.
$$
This is, of course, what a multivariate version of Lamperti's theorem
requires from the limit in any functional extremal theorem; see
e.g. Theorem 8.1.5 in \cite{samorodnitsky:2016}. 
}
\end{remark}
The following functional extremal theorem is the main result of this
section. 
\begin{theorem}  \label{thm; partial maxima pro conv}
 For all $0< \alpha < 2 $ and $0 < \beta_i < 1 $, $i=1,\ldots,d$, 
\begin{equation} \label{eq; partial maxima pro conv eq}
\left( \frac{1}{b_\bn} {M}_\bn (\bt), \, t  \in \mathbb{R}_+^d  \right )
\Rightarrow  \left(  \left( \frac{C_\alpha}{2} \right)^{ 1 / \alpha }
  W_{\alpha, \bbeta} (\bt), \, \bt \in  \mathbb{R}_+^d  \right )
\end{equation}
in the Skorohod $J_1$ topology on the space $D(\mathbb{R}_+^d)$. 
\end{theorem}
\begin{proof} 
The usual reference for multiparameter weak convergence is
\cite{straf:1972}. For our purposes there is little difference
between the properties of weak convergence in  $D(\mathbb{R}_+^d)$ for
$d=1$ and $d>1$. We will show weak convergence in $D([\0,\one])$,
and we use the series representation \eqref{eq; series rep}. By
\eqref{e:Mn.series} we can write, in law, 
\begin{equation} \label{e:Mn.series.t}
\frac{1}{b_\bn} {M}_\bn(\bt) = \max_{\0\leq \bk/\bn \leq\bt }
C_\alpha ^ {1 /\alpha } \sum _{ j=1} ^ \infty \epsilon_j 
\Gamma_j^{-1/ \alpha}   
\one_{A^d} \circ T^\bk(U_{j,\bn}),  \, \bt \in    [\0,\one]\,.
\end{equation}
We again use a truncation argument.  
For $\ell \in \mathbb{N}$, we define random fields ${M}_{\bn,\ell}$
by 
\begin{equation} \label{e:Mn.series.t.l}
\frac{1}{b_\bn} {M}_{\bn,\ell}(\bt) = \max_{\0\leq \bk/\bn \leq\bt }
C_\alpha ^ {1 /\alpha } \sum _{ j=1}^\ell \epsilon_j 
\Gamma_j^{-1/ \alpha}   
\one_{A^d} \circ T^\bk(U_{j,\bn}),  \, \bt \in   [\0,\one]\,.
\end{equation}
Similarly, starting with the truncated random sup measure
$\eta_{\alpha, \bbeta, \ell}$ we define a random field 
$W_{\alpha, \bbeta, \ell}$ by 
$$ 
W_{\alpha, \bbeta, \ell} (\bt) : = \eta_{\alpha, \bbeta, \ell} ( [\0, \bt] ),
\quad \bt \in  [\0,\one]\,.
$$
We start by proving that 
\begin{equation} \label{eq;maxima.conv.l}
\left( \frac{1}{b_\bn} {M}_{\bn,\ell} (\bt), \, t  \in \mathbb{R}_+^d  \right )
\Rightarrow  \left(  \left( \frac{C_\alpha}{2} \right)^{ 1 / \alpha }
  W_{\alpha, \bbeta,\ell} (\bt), \, \bt \in  [\0,\one]  \right )\,.
\end{equation}
Note that the representation \eqref{e:eta.l.1}  can be written, in law, as 
 \begin{equation}  \label{e:repr.Wl}
W_{\alpha , \bbeta, \ell } (\bt) =\max_{S \subset \{1, \ldots, \ell \}
  } \one_{ \{ I_S \cap [\0,\bt]\neq \emptyset \} }  \, 2^{1/\alpha}\sum_{j \in S }
 \one_{ \{\epsilon_j = 1 \} } \Gamma_j^{-1/\alpha}, \quad \bt \in [0,1]^d\,.
\end{equation}
Furthermore,  from the argument in Proposition \ref{prop; truncated random
  sup-mea conv} we know that outside of an event $A_\bn$ whose
probability 
goes to zero as $\bn\to\binfty$, the random field $\bigl( (1/b_\bn) M_{\bn,\ell} 
  (\bt), \, \bt \in [\0,\one]\bigr)$ coincides with the random field 
$$
    \max_{S\subset \{1, \ldots,   \ell \} } 
\one_{ \{ (\one/\bn)\hat{I}_{S,\bn} \cap [\0,\bt]
   \neq \emptyset \} }  C_\alpha^{1/\alpha} \sum_{j \in S }  
\one_{ \{\epsilon_j = 1 \} }  \Gamma_j^{-1/\alpha}, \quad \bt \in [\0,\one]\,;
$$
see Remark \ref{rk:alone}. Therefore, \eqref{eq;maxima.conv.l} will
follows once we prove that 
\begin{align} \label{eq;maxima.conv.l1}
&\left(  \max_{S\subset \{1, \ldots,   \ell \} } 
\one_{ \{ (\one/\bn)\hat{I}_{S,\bn} \cap [\0,\bt]
   \neq \emptyset \} }\sum_{j \in S }
 \one_{ \{\epsilon_j = 1 \} } \Gamma_j^{-1/\alpha}, \, \bt \in
  [\0,\one]   \right ) \\
\Rightarrow  &\left(   
  \max_{S \subset \{1, \ldots, \ell \}
  } \one_{ \{ I_S \cap [\0,\bt]\neq \emptyset \} }  \, \sum_{j \in S }
 \one_{ \{\epsilon_j = 1 \} } \Gamma_j^{-1/\alpha}, \, \bt \in
               [\0,\one]  \right )\,. \notag 
\end{align}
Since the Fell topology on $\mathcal{F}( [0,1]^d )$ is separable and
metrizable (see \cite{salinetti:wets:1981}), by the Skorohod
representation theorem, we can find a common probability space for
$\bigl( \hat{I}_{S,\bn}, \, S\subset \{1, \ldots,   \ell \}\bigr)$ and 
$\bigl(  I_s, \, S\subset \{1, \ldots,   \ell \}\bigr)$ such that the
convergence in Proposition \ref{lem; joint convergence lemma} becomes
the almost sure convergence. On that probability space we will prove
a.s. convergence in \eqref{eq;maxima.conv.l1}. 

For $i=1,\ldots, d$ and $S\subset \{1, \ldots,   \ell \}$ denote 
$$
t_{S,\bn}^{(i)} = \inf\bigl\{ t>0:\, t\in n_i^{-1}\hat{I}^{(i)}_{S,n_i}\bigr\}, \
\ \ t_{S}^{(i)} = \inf\bigl\{ t>0:\, t\in I^{(i)}_{S}\bigr\}\,.
$$
Then the a.s. convergence in Proposition \ref{lem; joint convergence
  lemma} implies that $t_{S,\bn}^{(i)}\to t_{S}^{(i)}$ a.s. as
$\bn\to\binfty$ for every $i=1,\ldots, d$ and $S\subset \{1, \ldots,
\ell \}$. If we denote $\bt_{S,\bn}=\bigl( t_{S,\bn}^{(1)}, \ldots,
t_{S,\bn}^{(d)}\bigr)$ and $\bt_{S}=\bigl( t_{S}^{(1)}, \ldots,
t_{S}^{(d)}\bigr)$ for $S\subset \{1, \ldots, \ell \}$, then
$\bt_{S,\bn}\to \bt_{S}$ a.s. as $\bn\to\binfty$. Since stable
subordinators do not hit fixed points, the $2^\ell$ points $ \bt_{S},
S\subset \{1, \ldots, \ell \}$ are distinct. Furthermore, given
$(\epsilon_j,\Gamma_j)_{j\in S}$, these points
determine the realization of the random field in the right hand side
of \eqref{eq;maxima.conv.l1}, while the $2^\ell$ points $ \bt_{S,\bn},
S\subset \{1, \ldots, \ell \}$ determine the realization of the random
field in the left hand side of \eqref{eq;maxima.conv.l1}. Therefore,
any homeomorphism of $[0,1]^d$ onto itself  that fixes the origin
and moves $ \bt_{S,\bn}$
to $ \bt_{S}$ for each $S\subset \{1, \ldots, \ell \}$ makes the
values of the field in the left hand side of \eqref{eq;maxima.conv.l1}
equal to the values of the field in the right hand side of
\eqref{eq;maxima.conv.l1}. The convergence $\bt_{S,\bn}\to \bt_{S}$
for each $S$ guarantees that these homeomorphisms can be chosen to
converge to the identity in the supremum norm. Hence a.s. convergence
in \eqref{eq;maxima.conv.l1}. 
  
To complete the proof we use, once again, the ``convergence together''
argument in Theorem 3.2 of \cite{billingsley:1999}. The first step to
this end is to show that 
\begin{equation} \label{e:bill.1}
\bigl( W_{\alpha, \bbeta, \ell} (\bt), \,
\bt \in [\0,\one]\bigr) \Rightarrow \bigl( W_{\alpha, \bbeta} (\bt), \,
\bt \in [\0,\one]\bigr)
\end{equation} 
as $\ell\to\infty$ in the Skorohod $J_1$ topology on the space
$D([\0,\one])$.  Since we can represent the random field in the
right hand side of \eqref{e:bill.1}, in law, as 
 \begin{equation}  \label{e:repr.W}
W_{\alpha , \bbeta} (\bt) =\sup_{S \subset \bbn} \one_{ \{ I_S \cap
  [\0,\bt]\neq \emptyset \} }  \, 2^{1/\alpha}\sum_{j \in S } 
 \one_{ \{\epsilon_j = 1 \} } \Gamma_j^{-1/\alpha}, \quad \bt \in
 [\0,\one]\,, 
\end{equation}
we will use the representations in law, \eqref{e:repr.Wl} and
\eqref{e:repr.W}, and prove a.s. convergence of the random fields in
the right hand sides of these representations. Furthermore, we will
show this a.s. convergence in the uniform distance. To this end, fix
$\bt\in [\0,\one]$ and 
note that by Proposition \ref{prop;intersection properties of prod
  shifted rsc}, with probability 1, any set  $S \subset \bbn$  such that 
$I_S \cap
  [\0,\bt]\not=\emptyset$ must be of cardinality smaller than
$\min_{i=1,\ldots, d}(1- \beta_i)^{-1} $.  Furthermore, for any such
$S$, the set $S\cap \{1,\ldots,\ell\}$ contributes to the maximum in
the right hand side in \eqref{e:repr.Wl}. Therefore, 
\begin{align*}
0 &\leq \sup_{S \subset \bbn} \one_{ \{ I_S \cap 
 [\0,\bt]\neq \emptyset \} }  \, 2^{1/\alpha}\sum_{j \in S } 
\one_{ \{\epsilon_j = 1 \} } \Gamma_j^{-1/\alpha} -
\max_{S \subset \{ 1,\ldots, \ell\}} \one_{ \{ I_S \cap
 [\0,\bt]\neq \emptyset \} }  \, 2^{1/\alpha}\sum_{j \in S } 
\one_{ \{\epsilon_j = 1 \} } \Gamma_j^{-1/\alpha} \\
&\leq 2^{1/\alpha} \min_{i=1,\ldots, d}(1- \beta_i)^{-1}
 \Gamma_{\ell+1}^{-1/\alpha}\,.
\end{align*}
That is,
\begin{align*}
&\sup_{\bt\in [\0,\one]} \left| \sup_{S \subset \bbn} \one_{ \{ I_S \cap
  [\0,\bt]\neq \emptyset \} }  \, 2^{1/\alpha}\sum_{j \in S } 
 \one_{ \{\epsilon_j = 1 \} } \Gamma_j^{-1/\alpha} -
\max_{S \subset \{ 1,\ldots, \ell\}} \one_{ \{ I_S \cap
  [\0,\bt]\neq \emptyset \} }  \, 2^{1/\alpha}\sum_{j \in S } 
 \one_{ \{\epsilon_j = 1 \} } \Gamma_j^{-1/\alpha}\right| \\
&\leq 2^{1/\alpha} \min_{i=1,\ldots, d}(1- \beta_i)^{-1}
  \Gamma_{\ell+1}^{-1/\alpha} \to 0
\end{align*}
as $\ell\to\infty$, proving the a.s. convergence in the uniform
distance.  

For the  second ingredient in the ``convergence together'' 
argument we use again the uniform distance and prove that for any
$\epsilon>0$, 
$$
\lim_{\ell \to \infty} \limsup_{\bn \to \binfty} 
P \left( \frac{1}{b_\bn} \sup_{\bt\in [\0,\one]}|{M}_\bn (\bt) -
  {M}_{\bn,\ell} (\bt)|>\epsilon\right) =0,.
$$
By \eqref{e:Mn.series.t} and \eqref{e:Mn.series.t.l} it is enough to
prove that 
$$
\lim_{\ell \to \infty} \limsup_{\bn \to \binfty} P \left(
\sup_{\bt\in [\0,\one]}\max_{\0\leq \bk/\bn \leq\bt }
 \left| \sum _{ j=\ell+1} ^ \infty \epsilon_j 
\Gamma_j^{-1/ \alpha}   
\one_{A^d} \circ T^\bk(U_{j,\bn})\right|>\epsilon \right)=0\,.
$$
This is, however, an immediate consequence of  
Proposition \ref{prop; tail sup-mea conv}.
\end{proof}

Theorem \ref{thm; partial maxima pro conv} has a natural counterpart
for the partial maxima of the absolute values of the random field (\ref{eq; SaS d-para
  random fields}). It can be obtained along the same lines as we
obtained Theorem \ref{thm; r s-m abs conv}. We omit the argument. 
 
\begin{theorem}  \label{thm; abs partial maxima pro conv}
Let $0< \alpha < 2 $ and $0 < \beta_i < 1 $, $i=1,\ldots,d$. Define 
\begin{equation*}
\tilde{M}_{\bn}(\bt) : =  \max_{ \mathbf{0} \leq \bk \leq \bn\bt }
\vert X_\bk \vert , \quad \bt \in \mathbb{R}_+^d\,.
\end{equation*}
Then 
\begin{equation} \label{eq; abs partial maxima pro conv eq}
\left( \frac{1}{b_\bn} \tilde{M}_\bn (\bt),\, \bt  \in \mathbb{R}_+^d
\right )  \Rightarrow  \left(   
\left(\frac{C_\alpha}{2}\right)^{1/\alpha}
\max\bigl( W^{(1)}_{\alpha,\bbeta} (\bt), W^{(2)}_{\alpha,\bbeta}
(\bt)\bigr),\, \bt \in  \mathbb{R}_+^d  \right ) 
\end{equation}
in the Skorohod $J_1$ topology on the space $D(\mathbb{R}_+^d)$. Here
$\bigl( W^{(1)}_{\alpha,\bbeta} (\bt), \, \bt \in  \mathbb{R}_+^d
\bigr)$ and $\bigl( W^{(2)}_{\alpha,\bbeta} (\bt), \, \bt \in  \mathbb{R}_+^d
\bigr)$ are two independent copies of the limiting process in Theorem
\ref{thm; partial maxima pro conv}. 
\end{theorem}

\begin{remark}{\rm 
The structure of the limit in Theorem \ref{thm; abs partial maxima pro
  conv} together with Remark \ref{rk:ss.field} immediately implies
that the random field in the right hand side of \eqref{eq; abs partial
  maxima pro conv eq} is self-similar, Furthermore, 
as in Remark \ref{rk:not.Fr}, in the case $0<\beta_i\leq 1/2$ for
some $i=1,\ldots, d$, and only in that case, the limiting random field
in Theorem \ref{thm; abs partial maxima pro conv} is Fr\'echet. In
this case an alternative way of stating the theorem is
$$
\left( \frac{1}{b_\bn} \tilde{M}_\bn (\bt),\, \bt  \in \mathbb{R}_+^d
\right )  \Rightarrow  \left(   
 C_\alpha^{1/\alpha}  W_{\alpha,\bbeta} (\bt),\, \bt \in  \mathbb{R}_+^d  \right )\,.
$$
}
\end{remark}

\bibliographystyle{authyear}
\bibliography{bibfile}

\end{document}